\documentclass[reqno]{amsart}
\usepackage{amssymb,latexsym}
\usepackage{enumerate}
\newtheorem{theorem}{Theorem}[section]
\newtheorem{lemma}[theorem]{Lemma}
\newtheorem{proposition}[theorem]{Proposition}
\newtheorem{corollary}[theorem]{Corollary}

\theoremstyle{definition}

\newtheorem{example}[theorem]{Example}

\theoremstyle{remark}
\newtheorem{remark}[theorem]{Remark}

\numberwithin{equation}{section}

\newcommand{\blankbox}[2]{}

\begin{document}

\title[The maximum diam theorem on Finsler manifolds]
{The maximum diam theorem on Finsler manifolds}

\author{Songting Yin}
\author{Qun He}

\thanks{This project is supported by
NNSFC(11471246) and AHNSF (1608085MA03).}


\date{}
\maketitle
\vspace{-6mm}
\begin{center}\footnotesize
Department of Mathematics and Computer Science,
 Tongling University, \\Tongling, 244000 Anhui, China;
Email: yst419@163.com\\
School of Mathematical Sciences, Tongji University,
Shanghai, 200092 China; \\Email: hequn@mail.tongji.edu.cn
\end{center}

\begin{abstract}
We prove that, for a Finsler space, if the weighted Ricci curvature is bounded below by a positive number
and the diam attains its maximal value, then it is isometric to a standard  Finsler sphere.
As an application, we show that  the first eigenvalue of the Finsler-Laplacian attains its lower bound
if and only if the Finsler manifold is isometric to a standard Finsler sphere, and moreover,
we obtain an explicit 1-st eigenfunction on the sphere.
\end{abstract}

\section*{ Introduction}
The classical Myers's theorem (\cite{M}) and Cheng's maximum diam theorem (\cite{Ch}) in Riemannian geometry are well known.
As for their generalizations, we can refer to \cite{BL}, \cite{Q} and \cite{R}.

In  Finsler geometry,  the Myers type theorem (\cite{Sh1}) states that if $(M,F,d\mu)$ is a complete, connected
Finsler $n$-manifold  such that $\textmd{Ric}\geq (n-1)k > 0$, then its diameter $Diam(M)\leq\frac{\pi}{\sqrt{k}}$.
Using the weighted Ricci curvature condition, Ohta (\cite{O1}) proved an analogue of
Myers's theorem (see Lemma 1.2 in Section 1 below). As in the case of Riemannian geometry, it is natural to ask what happens if the
diameter attains its maximal value. In \cite{KY}, the authors obtained the following Cheng type maximum diam theorem.

\vspace{1.5mm}
\hspace{-4.5mm}\textbf{Theorem A.} (\cite{KY}) \emph{Let $(M,F,d\mu)$ be a complete reversible connected Finsler $n$-manifold
with the Busemann-Hausdorff volume form. If the weighted Ricci curvature satisfies $\emph{Ric}_n\geq (n-1)k > 0$
and $Diam(M)=\frac{\pi}{\sqrt{k}}$,
then $(M,F)$ is isometric to the Euclidean sphere $\mathbb{S}^n(\frac{1}{\sqrt{k}})$.}
\vspace{1mm}

This result describes  the rigidity for the reversible Finsler manifolds, and then reduces to the Riemannian case.
However, as we know, most Finsler manifolds are not reversible, such as Randers spaces. We might, of course,
 wonder  whether or not there exist general Finsler manifolds attaining  the maximum diam.
 In this paper, we will give a positive answer to this question.

 Let $\mathbb{S}^n$ be an $n$-sphere equipped with a Finsler metric $F$ and a volume form $d\mu$.
If it has constant flag curvature $k$, vanishing $S$ curvature and
$Diam=\frac{\pi}{\sqrt{k}}$, then we call it a standard Finsler sphere
and denote it by $(\mathfrak{S}^n(\frac{1}{\sqrt{k}}),F,d\mu)$.
This definition is inspired by Bao-Shen's result (see Example 2.4 below).
Clearly, the Euclidean sphere $\mathbb{S}^n(\frac{1}{\sqrt{k}})$
 is certainly a standard Finsler sphere.  If $F$ is a Randers metric, we call it a standard Randers sphere denoted by
$(\mathcal{S}^n(\frac{1}{\sqrt{k}}),F,d\mu)$. Here $F$ is determined by navigation
data $(\mathfrak{g},W)$, where $\mathfrak{g}$ is the standard sphere metric and $W$ is a Killing vector.
The more details are shown in Section 2 below.

\vspace{1.5mm}
\hspace{-4.5mm}\textbf{Theorem 0.1.} \emph{Let $(M,F,d\mu)$ be a complete connected Finsler n-manifold
with the Busemann-Hausdorff volume form. If the weighted Ricci curvature satisfies $\emph{Ric}_n\geq(n-1)k>0$
and  $Diam(M)=\frac{\pi}{\sqrt{k}}$,
then $(M,F)$ is isometric to a standard Finsler sphere.
In particular, if $n\geq 3$ and F is an $(\alpha,\beta)$ metric,
then $(M,F)$ is isometric to a standard  Randers sphere.}
\vspace{1mm}

The weighted Ricci curvature $\textmd{Ric}_n\geq(n-1)k>0$ means that $\textmd{Ric}\geq(n-1)k>0$
and the $S$ curvature $S=0$ (see the definition in Sec.1 below).
The above definition of the standard Finsler sphere is reasonable. We will see that
in the Finsler setting the vanishing $S$ curvature is a necessary condition and
one can not deduce the maximum diam under a single condition of constant flag curvature $k$.
If $F$ is reversible, the Finsler sphere is just the standard Euclidean sphere $\mathbb{S}^n(\frac{1}{\sqrt{k}})$,
and if $F$ is not reversible, the Finsler spheres never exist singly but always in pairs.
Unfortunately,  we can not further characterize the standard Finsler sphere metric.
 This is because it is not clear, until now, about the classification for
 Finsler manifolds with constant flag curvature.

Theorem 0.1 covers Theorem A and shows that, apart from the Euclidean sphere, the maximum diam can be
attained by countless Finsler metrics on a sphere. In general,
if the  manifold has an arbitrary volume form $d\mu$ with a
suitable restriction, and the weighted Ricci curvature satisfies $\textmd{Ric}_N\geq(N-1)k>0$
for some real number $N\in[n,\infty)$, the conclusion still holds (see Theorem 2.1 and 2.8 below).
Since the Finsler manifolds discussed are not necessarily reversible, some methods in \cite{KY}
do not work and we have to explore  a new trail. For example, the reverse of a geodesic $\gamma$  is not necessarily a geodesic, the distance
 $d_F(p,q)\neq d_F(q,p)$, the gradient and Laplacian of a function $\nabla(-f)\neq-\nabla f,\Delta(-f)\neq-\Delta f$
in general, and so on. These lead to many difficulties in computation and  reasoning.
For this, we make full use of the reverse Finsler metric $\overleftarrow{F}$ which has many relationships with the metric $F$ on geodesic,
distance, gradient, Laplacian, curvature, and so on. This technique gives another way to deal with nonreversible Finsler manifolds (see
the proof for details in Theorem 2.1 below).

As an application, we can use it to characterize the rigidity of Finsler manifolds when the first eigenvalue
of Finsler Laplacian attains its lower bound. More precisely, we obtain Obata type rigidity theorem in
the following:

\vspace{1.5mm}
\hspace{-4.5mm}\textbf{Theorem 0.2.} \emph{Let $(M,F,d\mu)$ be a complete  connected Finsler n-manifold
with the Busemann-Hausdorff volume form. If the weighted Ricci curvature satisfies $\emph{Ric}_n\geq(n-1)k>0$,
then the first eigenvalue of Finsler-Laplacian $\lambda_1=nk$
if and only if  $(M,F)$ is isometric to a standard Finsler sphere.
In particular, if $n\geq 3$ and F is an $(\alpha,\beta)$ metric, then the equality holds if and only if  $(M,F)$ is isometric to a
standard Randers sphere.}
\vspace{1mm}

In \cite{YHS1}, the authors obtained Obata type rigidity theorem by using Theorem A. Since the condition is rather
strong, the manifold, which attains the lower bound of the first eigenvalue, reduces to the Euclidean sphere.
By contrast, Theorem 0.2 demonstrates that there are infinite non-Riemannian manifolds satisfying such a property. In addition,
we also construct the first eigenfunctions of the Finsler Laplacian(see the proof of Theorem 3.1 and 3.2 below).

The paper is organized as follows. In Section 1, some
fundamental concepts and formulas which are necessary for the present paper are
given, and some lemmas are contained. The maximum diam theorem and Obata type rigidity theorem
are then proved in Section 2  and Section 3, respectively.

\section{ Preliminaries}
Let $M$ be a smooth $n$-manifold and $\pi :TM\to M$
be the natural projection from the tangent bundle $TM$. Let
$(x,y)$ be a point of $TM$ with $x\in M$, $y\in T_xM$, and let
$(x^i,y^i)$ be the local coordinates on $TM$ with $y=y^i\partial
/\partial x^i$. A {\it Finsler metric} on $ M$ is a function
$F:TM\to [0,+\infty )$ satisfying the following properties:

(i){\it Regularity}: $F(x,y)$ is smooth in $TM\setminus 0$;

(ii){\it Positive homogeneity}: $F(x,\lambda y)=\lambda F(x,y)$
for $\lambda >0$;

(iii){\it Strong convexity}: The fundamental quadratic form
$$g:=g_{ij}(x,y)dx^i\otimes dx^j,\qquad g_{ij}:=\frac 12[F^2]_{y^i
y^j}$$
 is positively definite.

Let $X=X^{i}\frac{\partial}{\partial x^{i}}$ be a vector field. Then the \emph{covariant derivative}
of $X$  by $v\in T_{x}M$ with reference  vector $w\in T_{x}M\backslash 0$ is defined by
$$D^{w}_{v}X(x):=\left\{v^{j}\frac{\partial X^{i}}{\partial x^{j}}(x)
+\Gamma^{i}_{jk}(w)v^{j}X^{k}(x)\right\}\frac{\partial}{\partial x^{i}},$$
where $\Gamma^{i}_{jk}$ denote the coefficients of the Chern connection.

Given two linearly independent vectors $V,W\in T_{x}M\backslash0$, the flag curvature is defined by
$$K(V,W):=\frac{g_{V}(R^{V}(V,W)W,V)}{g_{V}(V,V)g_{V}(W,W)-g_{V}(V,W)^{2}},$$
where $R^{V}$ is the \emph{Chern curvature}:
$$R^{V}(X,Y)Z=D^{V}_{X}D^{V}_{Y}Z-D^{V}_{Y}D^{V}_{X}Z-D^{V}_{[X,Y]}Z.$$
 Then the Ricci curvature for $(M,F)$ is defined as
  $$\textmd{Ric}(V)=\sum_{\alpha=1}^{n-1} K(V,e_{\alpha}),$$
where $e_{1},\cdots,e_{n-1},\frac{V}{F(V)}$ form an orthonormal basis of $T_{x}M$ with respect to  $g_{V}$.

Let $(M,F,d\mu)$ be a Finsler $n$-manifold. Given a vector
$V\in T_xM$, let $\gamma :(-\varepsilon ,\varepsilon )\to M$ be a geodesic with $\gamma (0)=x,~\dot\gamma (0)=V$.
Define
$${\dot S}(V):=F^{-2}(V)\frac{d}{d t}[S(\gamma (t),\dot\gamma(t))]_{t=0},$$
where $S(V)$ denotes the $S$-curvature at $(x,V)$.
The {\it weighted Ricci curvature} of $(M,F,d\mu)$ is defined by (see \cite{O1})
$$\left\{\begin{array}{l}{\textmd{Ric}}_{n}(V):=\left\{\begin{array}{l}{\textmd{Ric}}(V)+\dot S(V),\quad{\rm for}~~S(V)=0,\\
-\infty,\qquad\qquad\qquad{\rm
otherwise},\end{array}\right.\\
{\textmd{Ric}}_{N}(V):={\textmd{Ric}}(V)+\dot S(V)-\frac {S(V)^{2}}{(N-n)F(V)^{2}},~~\forall~~N\in
   (n,\infty),\\
{\textmd{Ric}}_{\infty}(V):={\textmd{Ric}}(V)+\dot S(V),\end{array}\right.$$

For a smooth function $u$, the \emph{gradient vector} of $u$ at $x$ is defined by $\nabla u(x):=\mathcal{L}^{-1}(du)$,
where $\mathcal{L}:T_xM\to T_x^*M$ is the Legendre transform.
Let $V=V^{i}\frac{\partial}{\partial x^{i}}$ be a smooth vector field on $M$. The \emph{divergence} of $V$ with
respect to an arbitrary volume form $d\mu$ is defined by
$$ \textmd{div}V :=\sum_{i=1}^{n}\left(\frac{\partial V^{i}}{\partial x^{i}}+V^{i}\frac{\partial \Phi}{\partial x^{i}}\right),$$
where $d\mu =e^{\Phi}dx$. Then the \emph{Finsler-Laplacian} of $u$ can be defined by(\cite{GS})
$$\Delta u:=\textmd{div}(\nabla u), $$
 where the equality is in the weak $W^{1,2}(M)$ sense.

We remark here that, up to now, except for this definition, there have been several different definitions of
Finsler-Laplacians, introduced by Antonelli-Zastawniak (\cite{AL}), Bao-Lackey(\cite{BL2}), Barthelm(\cite{B}) and
Centroe(\cite{C}), respectively.

   Using the Finsler-Laplacian, Ohta obtained  the following comparison theorems under the weighted Ricci curvature condition.

\begin{lemma}\cite{OS}
 Let $(M,F,d\mu)$ be a Finsler n manifold. If the weighted Ricci curvature
satisfies $\emph{Ric}_{N}\geq (N-1)k>0,N\in[n,\infty)$,  then the Laplacian of the distance function $r(x)=d_{F}(p,x)$
from any given point $p\in M$ can be estimated as follows:
$$\Delta r\leq(N-1)\sqrt{k}\cot(\sqrt{k}r)$$
pointwise on $M\backslash (\{p\}\cup \textmd{Cut}(p))$  and in the sense of distributions on $M\backslash \{p\}$.
 \end{lemma}

 \begin{lemma}\cite{O1}
 Let $(M,F,d\mu)$ be a Finsler n manifold. If the weighted Ricci curvature satisfies $\emph{Ric}_{N}\geq (N-1)k>0,N\in[n,\infty)$,
 then $Diam M\leq\frac{\pi}{\sqrt{k}}$, and for any $0<r<R$, it holds that
 $$\max\left\{\frac{\emph{vol}^{d\mu}_FB^+_x(R)}{\emph{vol}^{d\mu}_FB^+_x(r)},\frac{\emph{vol}^{d\mu}_FB^-_x(R)}{\emph{vol}^{d\mu}_FB^-_x(r)}\right\}
\leq \frac{\int_0^{R}(\frac{\sin \sqrt{k}t}{\sqrt{k}})^{N-1}dt}{\int_0^{r}(\frac{\sin \sqrt{k}t}{\sqrt{k}})^{N-1}dt}.$$
\end{lemma}

\vspace{1mm}

\section {The maximum diam theorem}

We now consider the Finsler manifolds whose diam reach the maximum value. It is shown that they must
be of constant flag curvature and vanishing $S$ curvature. But it seems very difficult to  characterize the
manifolds further since the classification problem has not been solved so far.
There are two main barriers in our proof to overcome. Firstly, Finsler metrics are not reversible
in general. This brings about difficulties in some calculations and arguments. Here the use of the reversibility of $F$
to estimate some inequalities does not work any more. Some quantities should be precisely calculated. For this,
we have to make full use of the forward (resp. backward) geodesic (ball), the forward (resp. backward) distance
function, the reverse Finsler metric and  some corresponding comparison theorems. Secondly, we have to give the
constraint on the volume form $d\mu$, which is satisfied for the Busemann-Hausdorff volume form
if $N=n$. By means of this, we can prove that
for any point $p\in M$, there exists a point $q$ such that $d_F(p,q)$ attains the diam. Thus, by the arbitrariness
of the point $p$, we can further compute  the flag curvature and the $S$ curvature.

Let $(r,\theta)$ be the polar coordinate around $z\in M$ and write the volume form as $d\mu=\sigma_z(r,\theta)drd\theta$.
Set $S_xM:=\{y|y\in T_xM,F(y)=1\}$.  Then we can give

\begin{theorem}
Let $(M,F,d\mu)$ be a complete connected Finsler n-manifold.
If the weighted Ricci curvature and the volume form  satisfy
$\emph{Ric}_{N}\geq (N-1)k>0$, $\lim\limits_{r\to0}\int_{S_zM}\frac{\sigma_z(r,\theta)}{r^{N-1}}d\theta=C,\forall z\in M$
for some real number $C>0,N\in[n,\infty)$, and $Diam M=\frac{\pi}{\sqrt{k}}$,
then the flag curvature $K=k$, $S$ curvature $S=0$,  and $M$ is homeomorphic to the $n$-sphere $\mathbb{S}^{n}$.
\end{theorem}

\begin{remark}
Theorem 2.1 means that if $M$ attains its maximum diam, then $N=n$ and thus $\textmd{Ric}_{N}=\textmd{Ric}_{n}=\textmd{Ric}$.
In other words, for any fixed number $N>n$, the diam of $M$ can not achieve $\frac{\pi}{\sqrt{k}}$.
\end{remark}

\begin{proof}
Take $p,q\in M$ such that $d_F(p,q)=\frac{\pi}{\sqrt{k}}$. Let $r_p^+(x)=d_F(p,x)$ be the forward distance function from $p$
and $r_q^-(x)=d_F(x,q)$ be the backward distance function from $q$. For any point $x\in M$, we claim:
\begin{align}\label{1}
d_F(p,x)+d_F(x,q)=d_F(p,q)=\frac{\pi}{\sqrt{k}}.
\end{align}
If not, then there exists a number $\varepsilon>0$ such that $d_F(p,x)+d_F(x,q)=\frac{\pi}{\sqrt{k}}+2\varepsilon.$
Denote by $r_1=d_F(p,x)-\varepsilon$ and $r_2=d_F(x,q)-\varepsilon$. It is clear that $r_1>0,r_2>0$. Then $B_p^+(r_1)$, $B_q^-(r_2)$ and
$B_x^{\pm}(\varepsilon)$ are pairwise disjoint, where $B_p^+(r_1)$ and $B_q^-(r_2)$ are
the forward (resp. backward) geodesic ball centered at $p$ (resp. $q$) of radius $r_1$ (resp. $r_2$) and
$B_x^{\pm}(\varepsilon)=B_x^{+}(\varepsilon)\cap B_x^{-}(\varepsilon)$.
In fact, if $y\in B_p^+(r_1)\cap B_q^-(r_2)$, which means that $d_F(p,y)< r_1$ and $d_F(y,q)< r_2$, then
$$d_F(p,q)\leq d_F(p,y)+d_F(y,q) < r_1+r_2=d_F(p,x)+d_F(x,q)-2\varepsilon=\frac{\pi}{\sqrt{k}}.$$
This is a contradiction. On the other hand, if $y\in B_p^+(r_1)\cap B_x^{\pm}(\varepsilon)$,
which means that $d_F(p,y)<r_1$ and $d_F(y,x)< \varepsilon, d_F(x,y)<\varepsilon $, then
$$ r_1+ \varepsilon=d_F(p,x)\leq d_F(p,y)+d_F(y,x)<r_1+ \varepsilon$$
which is also a contradiction. By a similar argument, we can further conclude that
$B_q^-(r_2)\cap B_x^{\pm}(\varepsilon)=\emptyset$.
Next by volume comparison theorem (Lemma 1.2), we have
\begin{align}\label{2}
1&=\frac{\textmd{vol}_F^{d\mu}(M)}{\textmd{vol}_F^{d\mu}(M)}\geq
\frac{\textmd{vol}_F^{d\mu}B_p^+(r_1)+\textmd{vol}_F^{d\mu}B_q^-(r_2)+
\textmd{vol}_F^{d\mu}(B_x^{\pm}(\varepsilon))}{\textmd{vol}_F^{d\mu}(M)}\nonumber\\
&\geq\frac{\int_0^{r_1}(\frac{\sin \sqrt{k}t}{\sqrt{k}})^{N-1}dt+
\int_0^{r_2}(\frac{\sin \sqrt{k}t}{\sqrt{k}})^{N-1}dt}{\int_0^{\frac{\pi}{\sqrt{k}}}(\frac{\sin \sqrt{k}t}{\sqrt{k}})^{N-1}dt}
+\frac{\textmd{vol}_F^{d\mu}(B_x^{\pm}(\varepsilon))}{\textmd{vol}_F^{d\mu}(M)}.
\end{align}
Note that $r_1+r_2=\frac{\pi}{\sqrt{k}}$. Hence,
$$\int_0^{r_1}\left(\frac{\sin \sqrt{k}t}{\sqrt{k}}\right)^{N-1}dt
=\int_{r_2}^{\frac{\pi}{\sqrt{k}}}\left(\frac{\sin \sqrt{k}t}{\sqrt{k}}\right)^{N-1}dt.$$
Therefore, (\ref{2}) can be rewritten as follows
$$1\geq1+\frac{\textmd{vol}_F^{d\mu}(B_x^{\pm}(\varepsilon))}{\textmd{vol}_F^{d\mu}(M)}$$
which implies that $\varepsilon=0$. This contradicts the assumption above that $\varepsilon>0$.

From (\ref{1}) we obtain
$$r_p^+(x)+r_q^-(x)=\frac{\pi}{\sqrt{k}},$$
which gives
$$\Delta r_p^+(x)=\Delta\left(\frac{\pi}{\sqrt{k}}-r_q^-(x)\right)=\Delta(-r_q^-(x))=-\overleftarrow{\Delta}r_q^-(x),$$
where $\overleftarrow{\Delta}$ denotes the Laplacian of Finsler metric $\overleftarrow{F}(x,y):=F(x,-y)$.
Further, for the reverse Finsler metric $\overleftarrow{F}$,
$\overleftarrow{d}_{\overleftarrow{F}}(p,q)=d_{F}(q,p)$, $\overleftarrow{\nabla} u=-\nabla(-u)$
and $\overleftarrow{\textmd{Ric}}_N(x,y)=\textmd{Ric}_N(x,-y)$. Since $\textmd{Ric}_N(x,y)\geq(N-1)k,\forall y\in T_xM$, then
$\overleftarrow{\textmd{Ric}}_N(x,y)\geq(N-1)k,\forall y\in T_xM$. Thus by the Laplacian comparison theorem
(Lemma 1.1)of the revised version for the reverse Finsler metric $\overleftarrow{F}$, we have
\begin{align}
(N-1)\sqrt{k}\cot(\sqrt{k}r_p^+(x))&\geq\Delta r_p^+(x)= -\overleftarrow{\Delta}(r_q^-(x))\nonumber\\
&\geq-(N-1)\sqrt{k}\cot(\sqrt{k}r_q^-(x))\nonumber\\
&=(N-1)\sqrt{k}\cot(\sqrt{k}r_p^+(x))\nonumber
\end{align}
which yields
\begin{align}\label{a0}
\Delta r_p^+(x)=(N-1)\sqrt{k}\cot(\sqrt{k}r_p^+(x)).
\end{align}
In the following, we write $r$ instead of $r_p^+(x)$ for simplicity. Direct computation gives
\begin{align}\label{3}
\frac{\partial}{\partial r}(\Delta r)+\frac{(\Delta r)^2}{N-1}=-(N-1)k.
\end{align}
Let $S_{p}(r(x))$ be the forward geodesic sphere of radius
$r(x)$ centered at $p$. Choosing the local $g_{\nabla r}$-orthonormal frame
$E_{1},\cdots,E_{n-1}$ of $S_{p}(r(x))$ near $x$, we get local vector fields
$E_{1},\cdots,E_{n-1},E_{n}=\nabla r$ by parallel transport
along geodesic rays. Thus, it follows from \cite{WX} that
\begin{align}\label{4}
\frac{\partial}{\partial r}\textmd{tr}_{\nabla r}H(r)=-\textmd{Ric}(\nabla r)-\sum_{i,j}[H(r)(E_i,E_j)]^2,
\end{align}
where $H(r)$ is the Hessian of distant function $r$.
On the other hand, one has  (\cite{WX}),
\begin{align}\label{5}
\Delta r=\textmd{tr}_{\nabla r}H(r)-S(\nabla r).
\end{align}
Therefore, from (\ref{3})-(\ref{5}), we derive
\begin{align}\label{6}
-(N-1)k&=\frac{\partial}{\partial r}(\Delta r)+\frac{(\Delta r)^2}{N-1}\nonumber\\
&=\frac{\partial}{\partial r}(\textmd{tr}_{\nabla r}H(r)-S(\nabla r))+\frac{1}{N-1}(\textmd{tr}_{\nabla r}H(r)-S(\nabla r))^2\nonumber\\
&\leq\frac{\partial}{\partial r}\textmd{tr}_{\nabla r}H(r)-\dot{S}(\nabla r)
+\frac{1}{n-1}(\textmd{tr}_{\nabla r}H(r))^2+\frac{S(\nabla r)^2}{N-n}\nonumber\\
&\leq\frac{\partial}{\partial r}\textmd{tr}_{\nabla r}H(r)-\dot{S}(\nabla r)+\sum_{i,j}[H(r)(E_i,E_j)]^2+\frac{S(\nabla r)^2}{N-n}\nonumber\\
&=-\textmd{Ric}(\nabla r)-\dot{S}(\nabla r)+\frac{S(\nabla r)^2}{N-n}\nonumber\\
&=-\textmd{Ric}_N(\nabla r)\leq-(N-1)k,
\end{align}
where the first inequality holds from the following by replacing $a=\textmd{tr}_{\nabla r}H(r),b=S(\nabla r)$:
\begin{align}\label{7}
\frac{(a-b)^2}{N-1}&=\frac{a^2}{n-1}+\frac{b^2}{N-n}-\frac{N-n}{(n-1)(N-1)}(a+\frac{n-1}{N-n}b)^2\nonumber\\
&\leq\frac{a^2}{n-1}+\frac{b^2}{N-n}.
\end{align}
Using (\ref{6}) and (\ref{7}), we obtain
\begin{align}
\left\{
\begin{array}{ll}
   &\frac{\textmd{tr}_{\nabla r}H(r)}{n-1}=\frac{-S(\nabla r)}{N-n}=\frac{\Delta r}{N-1},\\
   &\\
    & \sum_{i,j}[H(r)(E_i,E_j)]^2=\frac{1}{n-1}(\textmd{tr}_{\nabla r}H(r))^2.\nonumber
\end{array}\right.
\end{align}
Thus,
\begin{align}\label{8}
\nabla^2r(E_i,E_j)&:=H(r)(E_i,E_j)\nonumber\\
&=\left\{
\begin{array}{ll}
   \frac{\textmd{tr}_{\nabla r}H(r)}{n-1}=\frac{-S(\nabla r)}{N-n}=\frac{\Delta r}{N-1}=\sqrt{k}\cot(\sqrt{k}r),&i=j<n \\
    0,&i\neq j.
\end{array}\right.
\end{align}
Now we calculate the flag curvature of $(M,F)$. By (\ref{8}) we observe that $\{E_i\}_{i=1}^{n-1}$ are $(n-1)$ eigenvectors of
$\nabla^2r$. That is,
$$D^{\nabla r}_{E_i}\nabla r=\sqrt{k}\cot(\sqrt{k}r)E_i,\qquad i=1,\cdots,n-1.$$
Since $\nabla r$ is a geodesic field on $(M,F)$, the flag curvature $K(\nabla r;\cdot)$ is equal to the sectional
curvature of the weighted Riemannian manifold  $(M,g_{\nabla r})$.
Note that $\{E_i\}_{i=1}^{n-1}$ are $(n-1)$ eigenvectors of $\nabla^2r$  and parallel along  the geodesic ray.
By a straightforward computation, we get, for $1\leq i\leq n-1$,
\begin{align}
K(\nabla r;E_i)&=R^{\nabla r}(E_i,\nabla r,E_i,\nabla r)=g_{\nabla r}(R^{\nabla r}(E_i,\nabla r)\nabla r,E_i)\nonumber\\
&=g_{\nabla r}(D^{\nabla r}_{E_i}D^{\nabla r}_{\nabla r}\nabla r-D^{\nabla r}_{\nabla r}D^{\nabla r}_{E_i}\nabla r
-D^{\nabla r}_{[E_i,\nabla r]}\nabla r,E_i)\nonumber\\
&=-g_{\nabla r}(D^{\nabla r}_{\nabla r}(\sqrt{k}\cot(\sqrt{k}r))E_i+
D^{\nabla r}_{D^{\nabla r}_{E_i}\nabla r-D^{\nabla r}_{\nabla r}E_i}\nabla r,E_i)\nonumber\\
&=-g_{\nabla r}(-k\csc^2(\sqrt{k}r)E_i+D^{\nabla r}_{\sqrt{k}\cot(\sqrt{k}r)E_i}\nabla r,E_i)\nonumber\\
&=k\csc^2(\sqrt{k}r)-\sqrt{k}\cot(\sqrt{k}r)g_{\nabla r}(D^{\nabla r}_{E_i}\nabla r,E_i)\nonumber\\
&=k\csc^2(\sqrt{k}r)-k\cot^2(\sqrt{k}r)\nonumber\\
&=k.\nonumber
\end{align}

We have proved that for any $x\in M$, $K(x,\nabla r;\cdot)=k$, where $r$ is the distance function from $p$. Next we will
prove that along any direction $V\in T_xM$, $K(x,V;\cdot)=k$, which yields $K\equiv k$. For this, we will vindicate that
for any fixed point $ p'\in M$ there exists a point $q'\in M$ such that $d_F(p',q')=\frac{\pi}{\sqrt{k}}$.
If this is not true, then there is a small $\varepsilon>0$ such that $B^+_{p'}(\frac{\pi}{\sqrt{k}}-\varepsilon)=M$. Define
$$f(x,r):=\frac{\textmd{vol}_F^{d\mu}(B^+_{x}(r))}{\int_0^r(\frac{\sin\sqrt{k}t}{\sqrt{k}})^{N-1}dt}.$$
Then
\begin{align}\label{9}
 f(p',\frac{\pi}{\sqrt{k}}-\varepsilon)&=\frac{\textmd{vol}_F^{d\mu}(B^+_{p'}(\frac{\pi}{\sqrt{k}}-\varepsilon))}
 {\int_0^{\frac{\pi}{\sqrt{k}}-\varepsilon}(\frac{\sin\sqrt{k}t}{\sqrt{k}})^{N-1}dt}
= \frac{\textmd{vol}_F^{d\mu}(M)}{\int_0^{\frac{\pi}{\sqrt{k}}-\varepsilon}(\frac{\sin\sqrt{k}t}{\sqrt{k}})^{N-1}dt}\nonumber\\
 &=\frac{\textmd{vol}_F^{d\mu}(B^+_{p}(\frac{\pi}{\sqrt{k}}-\varepsilon))+\textmd{vol}_F^{d\mu}(B^-_{q}(\varepsilon))}
 {\int_0^{\frac{\pi}{\sqrt{k}}-\varepsilon}(\frac{\sin\sqrt{k}t}{\sqrt{k}})^{N-1}dt}\nonumber\\
 &>\frac{\textmd{vol}_F^{d\mu}(B^+_{p}(\frac{\pi}{\sqrt{k}}-\varepsilon))}
 {\int_0^{\frac{\pi}{\sqrt{k}}-\varepsilon}(\frac{\sin\sqrt{k}t}{\sqrt{k}})^{N-1}dt}
 = f(p,\frac{\pi}{\sqrt{k}}-\varepsilon),
 \end{align}
where the third equality is due to (\ref{1}).
Since $\textmd{Ric}_N\geq (N-1)k$, by Laplacian comparison theorem (Lemma 1.1), we have
\begin{align}\label{a1}
\Delta r\leq(N-1)\sqrt{k}\cot(\sqrt{k}r)=(N-1)\frac{(\sin\sqrt{k}r)'}{\sin\sqrt{k}r},
 \end{align}
where $r$ is the distance function from any fixed point  $z\in M$.
Let $(r,\theta)$ be the polar coordinates of $x$. Then
$r(x)=F(v),\theta^{\alpha}(x)=\theta^{\alpha}(\frac{v}{F(v)})$ and $v=\exp^{-1}_{z}(x)$.
Thus, the above inequality shows
\begin{align}\label{a2}
\frac{\partial}{\partial r}\log\sigma_z\leq\frac{\partial}{\partial r}\log\tilde{\sigma},
\qquad \tilde{\sigma}:=\left(\frac{\sin\sqrt{k}r}{\sqrt{k}}\right)^{N-1}.
 \end{align}
Integrating both sides gives
\begin{align}\label{a3}
\frac{\sigma_z(r,\theta)}{\tilde{\sigma}(r)}\leq\frac{\sigma_z(\delta,\theta)}{\tilde{\sigma}(\delta)}(\delta\to0).
 \end{align}
Note that if $r$ is the distance function from $p$, then, by (\ref{a0}), the equalities holds in (\ref{a1})-(\ref{a3}).
 Hence, using the condition on $d\mu$ of Theorem 2.1, we have
\begin{align}
 f(p,\frac{\pi}{\sqrt{k}}-\varepsilon)
 &= \lim_{r\to0}\frac{\textmd{vol}_F^{d\mu}(B^+_{p}(r))}
 {\int_0^{r}\tilde{\sigma}dt}
 =\lim\limits_{r\to0}\int_{S_pM}\frac{\sigma_p(r,\theta)}{r^{N-1}}d\theta
 =C,\nonumber\\
 f(p',\frac{\pi}{\sqrt{k}}-\varepsilon)
 &\leq \lim_{r\to0}\frac{\textmd{vol}_F^{d\mu}(B^+_{p'}(r))}
 {\int_0^{r}\tilde{\sigma}dt}
 =\lim\limits_{r\to0}\int_{S_{p'}M}\frac{\sigma_{p'}(r,\theta)}{r^{N-1}}d\theta
 =C,\nonumber
 \end{align}
which contradicts to (\ref{9}).

In what follows, we prove that along any direction $V\in T_xM$, $K(x,V;\cdot)=k,\forall x\in M$.
First, we draw  a minimal geodesic $\overleftarrow{\eta}$ of reverse Finsler metric
$\overleftarrow{F}$  satisfying $\overleftarrow{\eta}(0)=x,\overleftarrow{\eta}'(0)=\frac{-V}{F(V)}$.
Then its reverse $\eta$ is a normal minimal geodesic of $F$ satisfying $\eta(0)=x,\dot{\eta}(0)=\frac{V}{F(V)}$.
Choose $p'= \overleftarrow{\eta}(\delta)$ for some small $\delta>0$. Then $d_F(p',x)=L(\eta_{\widehat{p'x}})=\delta$.
Second, let $q'$ be the point such that $d_F(p',q')=\frac{\pi}{\sqrt{k}}$ and draw a minimal geodesic
$\tilde{\gamma}$ from $x$ to $q'$. Then, by (\ref{1}), we see $\eta|_{\widehat{p'x}}\cup\tilde{\gamma}$ is a minimal geodesic from $p'$
to $q'$. Thus, by the same argument, we obtain $K(x,V;\cdot)=k$.

Now we are to prove $(M,F,d\mu)$ has vanishing $S$ curvature. From (\ref{8}) we have
$$S(x,\nabla r)=-(N-n)\sqrt{k}\cot(\sqrt{k}r),$$
where $r(x)=d_F(p,x)$. Choose an arbitrary point $p'$ on the minimal geodesic $\gamma$ from $p$ to
$x$ and let $q'$ be the point such that $d_F(p',q')=\frac{\pi}{\sqrt{k}}$. Then $\gamma|_{p'x}$ is
the minimal geodesic from $p'$ to $x$ and we can extend $\gamma$ to pass through $q'$.
Write $r_1:=d_F(p,x),r_2:=d_F(p',x)$. Then $\nabla r_1=\nabla r_2$ and
$$-(N-n)\sqrt{k}\cot(\sqrt{k}r_1)=S(x,\nabla r_1)=S(x,\nabla r_2)=-(N-n)\sqrt{k}\cot(\sqrt{k}r_2).$$
Since $r_1\neq r_2$, we have $N=n$ which yields $S(x,\nabla r)=0$. By the arbitrariness of choice of
the points $p,p'$ and $x$, for any vector $V\in T_xM$, we can choose a suitable geodesic $\gamma$ passing through
$x$ and satisfying $\nabla r(x)=\frac{V}{F(V)}$. This gives $S(x,V)=0,\forall x\in M,\forall V\in T_xM$.

Let $p,q\in M$ be the points as above. Then $d_F(p,q)=\frac{\pi}{\sqrt{k}}$. It follows from (\ref{1}) that, for any
point $x$, there exists a minimal geodesic $\gamma$ from $p$ to $q$ passing through $x$.  This means that $x$ is not
cut point of $p$. If not, there are two minimal geodesics $\eta_1$ and $\eta_2$ from $p$ to $x$. Then
$\eta_1\cup\gamma_{\widehat{xq}}$ and $\eta_2\cup\gamma_{\widehat{xq}}$ are two minimal geodesics from $p$ to $q$ passing through $x$.
This is  impossible. Therefore, by arbitrariness of choice of point $x$, we conclude $q=Cut(p)$.
Thus
$$\exp_{p}:T_{p}M\supset \textbf{B}_{p}(\frac{\pi}{\sqrt{k}})\longrightarrow M^{n}\backslash\{q\}$$
 is a diffeomorphism. On the other hand,
 $$\exp_{\tilde{p}}:T_{\tilde{p}}\mathbb{S}^{n}\supset \textbf{B}_{\tilde{p}}(\pi)\longrightarrow \mathbb{S}^{n}\backslash\{\tilde{q}\}$$
is also a diffeomorphism, where $\mathbb{S}^{n}$ is the $n$-sphere, $\tilde{p},\tilde{q}$ are the south pole and north pole respectively. Let $(\tilde{r},\tilde{\theta}^{\alpha})$ be the polar coordinate system of $T_{\tilde{p}}\mathbb{S}^{n}$ and $(r,\theta^{\alpha})$
be  the polar coordinate system of $T_{p}M$. Define $h:T_{\tilde{p}}\mathbb{S}^{n}\longrightarrow T_{p}M$ by $r=\frac{\tilde{r}}{\sqrt{k}},\theta^{\alpha}=\tilde{\theta}^{\alpha}$. Then $h$ is a diffeomorphism.
Now we define $\psi:M^{n}\longrightarrow \mathbb{S}^{n}$ by
$$\psi(x)=\left\{\begin{array}{cc}
            \exp_{\tilde{p}}\circ h^{-1}\circ \exp_{p}^{-1}(x) & x\neq q \\
            \tilde{q} & x=q
          \end{array}\right.$$
Obviously, $\psi$ is homeomorphic. That is, $M$ is homeomorphic to $\mathbb{S}^{n}$.

\end{proof}

In \cite{KY}, the authors obtained the maximum diam theorem for the reversible Finsler manifolds
by using the condition of Ricci curvature $\textmd{Ric} \geq(n-1)k>0$ and vanishing $S$ curvature.
By the weighted Ricci curvature defined in Section 1 above, the condition can
 be also written as $\textmd{Ric}_n\geq(n-1)k>0$ (see Theorem A).
Note that a reversible Finsler sphere is actually the Euclidean sphere (see Remark 0.1 above).
Then, from Theorem 2.1, we  generalize  Theorem A as follows.

\begin{corollary}
Let $(M,F,d\mu)$ be a complete connected Finsler n-manifold with the Beausemann-hausdorff volume form.
If the weighted Ricci curvature satisfies $\emph{Ric}_n\geq(n-1)k>0$ and  $Diam(M)=\frac{\pi}{\sqrt{k}}$,
then $(M,F)$ is isometric to a standard Finsler sphere.
\end{corollary}

\begin{proof}
Note that the Busemann-Hausdorff volume form satisfies
$\lim\limits_{r\to0}\frac{\textmd{vol}^{d\mu}_F(B^+_x(r))}{\textmd{vol}\mathbb{B}^n(r)}=1,\forall x\in M$,
where $B_x^+(r)$ is the forward geodesic ball of $M$ and  $\mathbb{B}^n(r)$ is the Euclidean ball (\cite{Sh2}).
In this case, the condition on $d\mu$ in Theorem 2.1 is satisfied for $N=n$.
From Theorem 2.1,we have $K=k$ and $M$ is homeomorphic to $\mathbb{S}^n$. Thus we can view $(M,F,d\mu)$ as a
sphere  with constant flag curvature and vanishing $S$ curvature. Hence, it is a standard Finsler sphere.
\end{proof}

As is well known that there are infinite nonreversible Finsler metrics with constant flag curvature on the sphere
 $\mathbb{S}^n$. Since these metrics have not been classified completely, we can not characterize the manifolds
 when the diam attains its maximum. However, the following example shows that the maximum diam can be achieved
 in non-Riemannian case.
\begin{example} \cite{BS}
View $\mathbb{S}^3$ as a compact Lie group. Let $\zeta^1,\zeta^2,\zeta^3$ be the standard right invariant 1-form
on $\mathbb{S}^3$ satisfying
$$d\zeta^1=2\zeta^2\wedge\zeta^3,\quad d\zeta^2=2\zeta^3\wedge\zeta^1,\quad d\zeta^3=2\zeta^1\wedge\zeta^2.$$
For $k\geq1$, define
$$\alpha_k(y)=\sqrt{(k\zeta^1(y))^2+k(\zeta^2(y))^2+k(\zeta^3(y))^2},\quad \beta_k(y)=\sqrt{k^2-k}\zeta^1(y).$$
Then $F_k=\alpha_k+\beta_k$ is a Randers metric on $\mathbb{S}^3$ satisfying
$$K\equiv1,\quad S\equiv0,\quad Diam(\mathbb{S}^3,F_k)=\pi.$$
\end{example}

\vspace{3mm}
In what follows, we focus on the Randers spaces.
Let $\mathfrak{g}$ be the standard sphere metric and $W=W^i\frac{\partial}{\partial x^i}$ be a
Killing vector field on $\mathbb{S}^n(\frac{1}{\sqrt{k}})$. Then the sectional
curvature $K_{\mathfrak{g}}=k$.  Define a Randers metric by
\begin{align}\label{b}
F=\frac{\sqrt{\lambda \mathfrak{g}^2+W_0^2}}{\lambda}-\frac{W_0}{\lambda},
\end{align}
where  $\lambda=1-\|W\|^2_{\mathfrak{g}}$. Then the sphere $\mathbb{S}^n(\frac{1}{\sqrt{k}})$ is
equipped with a Randers metric $F$ of constant flag  curvature $k$ (see \cite{BCS},\cite{BRS} for details).
We say it a \emph{standard Randers sphere} and denote it by $\mathcal{S}^n(\frac{1}{\sqrt{k}})$.

\begin{proposition}
On a standard Randers sphere $(\mathcal{S}^n(\frac{1}{\sqrt{k}}),F,d\mu)$ with the Busemann-Hausdorff volume form, we have
\begin{enumerate}
  \item $S=0$;
  \item $\emph{vol}^{d\mu}_F(\mathcal{S}^n(\frac{1}{\sqrt{k}}))=\emph{vol}_\mathfrak{g}(\mathbb{S}^n(\frac{1}{\sqrt{k}}))$;
  \item $Diam(\mathcal{S}^n(\frac{1}{\sqrt{k}}),F)=\frac{\pi}{\sqrt{k}}$.
\end{enumerate}
Clearly,  $(\mathcal{S}^n(\frac{1}{\sqrt{k}}),F,d\mu)$ is naturally a standard Finsler sphere.
\end{proposition}

\begin{proof} In (\ref{b}), $W$ is a Killing vector field. This is equivalent to $S=0$ (see \cite{BCS}).
Since $d\mu$ is the Busemann-Hausdorff volume form, we know  that $dV_F = dV_\mathfrak{g}$. Thus,
$$\textmd{vol}^{d\mu}_F(\mathcal{S}^n(\frac{1}{\sqrt{k}}))=\textmd{vol}_\mathfrak{g}(\mathbb{S}^n(\frac{1}{\sqrt{k}}))$$
Now fix $p\in \mathcal{S}^n(\frac{1}{\sqrt{k}})$. Using $K=k$ and Theorem 18.3.1 in \cite{Sh1}, there exists
$q\in \mathcal{S}^n(\frac{1}{\sqrt{k}})$ such that
$$\exp_p(\frac{\pi}{\sqrt{k}}\xi)=q,\qquad \forall\xi\in S_p(\mathcal{S}^n(\frac{1}{\sqrt{k}})),$$
where $S_p(\mathcal{S}^n(\frac{1}{\sqrt{k}})):=\{v|v\in T_p(\mathcal{S}^n(\frac{1}{\sqrt{k}})),F(v)=1\}$.
From the proof of the volume comparison
theorem (\cite{Sh2}, or Theorem 16.1.1, p.250, \cite{Sh1}),
$$\textmd{vol}^{d\mu}_F(B^+_p(r))\leq\sigma_n(r),$$
where $\sigma_n(r)$ denotes the volume of the metric ball of radius $r$ in $\mathbb{S}^n(\frac{1}{\sqrt{k}})$.
The equality holds if and only if $B^+_p(r)\subset \mathcal{D}_p$, i.e., $\mathbf{i}_p\geq r$.
By the Bonnet-Myers theorem, $Diam(\mathcal{S}^n(\frac{1}{\sqrt{k}}))\leq\frac{\pi}{\sqrt{k}}$, which means
 $\overline{B^+_p(\frac{\pi}{\sqrt{k}})}=\mathcal{S}^n(\frac{1}{\sqrt{k}})$. Therefore,
$$\textmd{vol}^{d\mu}_F(B^+_p(\frac{\pi}{\sqrt{k}}))=\textmd{vol}^{d\mu}_F(\mathcal{S}^n(\frac{1}{\sqrt{k}}))
=\textmd{vol}_\mathfrak{g}(\mathbb{S}^n(\frac{1}{\sqrt{k}}))=\sigma_n(\frac{\pi}{\sqrt{k}}).$$
We deduce that $\mathbf{i}_p\geq \frac{\pi}{\sqrt{k}}$, which yields $d_F(p,q)=\frac{\pi}{\sqrt{k}}$.
\end{proof}

\begin{remark}
In \cite{Sh3}, the author studied the reversible Finsler manifolds with constant flag curvature.
For the nonreversible case, We show in Proposition 2.5 that there are infinite Randers metrics on
 $\mathcal{S}^n(\frac{1}{\sqrt{k}})$ with constant flag
curvature $K=k$ and vanishing $S$ curvature. Moreover, they have the same diameter and volume
as the Euclidean sphere, but they are not necessarily isometric to each other.
Write $F\triangleq(\mathfrak{g},W)$ if $F$ is expressed by (\ref{b}). Set
$$\mathfrak{F}:=\{F|F\triangleq(\mathfrak{g},W), W \textmd{ is a Killing vector with } \|W\|_{\mathfrak{g}}<1.\}$$
Then $\mathfrak{F}$ determines all standard  Randers spheres $(\mathcal{S}^n(\frac{1}{\sqrt{k}}),F)$,
and especially includes $(\mathbb{S}^n(\frac{1}{\sqrt{k}}),\mathfrak{g})$. Fix a Killing vector field $W$ and
let $W_a:=aW,a\in[0,\frac{1}{\|W\|_{\mathfrak{g}}})$. Then each $W_a$ is also a Killing vector field satisfying $\|W_a\|_{\mathfrak{g}}<1$
and $\big\{(\mathcal{S}^n(\frac{1}{\sqrt{k}}),F\triangleq(\mathfrak{g},W_a))\big\}$
 make up a family of standard  Randers spheres.

\end{remark}

\begin{remark}
Let $F=\alpha\phi(\frac{\beta}{\alpha})$ be an $(\alpha,\beta)$ metric.
Then, by Theorem 1.1 in \cite{CST}, a standard $(\alpha,\beta)$ sphere is actually a standard Randers sphere
if $\phi$  is a polynomial. Moreover, when $n\geq3$, then, by Theorem 0.4 in \cite{ZH}, every standard $(\alpha,\beta)$
sphere is  standard  Randers sphere for all $\phi$.
\end{remark}

\begin{theorem}
Let $(M,F,d\mu)$ be a complete connected Randers n-manifold. If the weighted Ricci curvature
and the volume form satisfy $\emph{Ric}_N\geq(N-1)k>0,$
$\lim\limits_{r\to0}\int_{S_zM}\frac{\sigma_z(r,\theta)}{r^{N-1}}d\theta=C, \forall z\in M$ for
some real number $C>0,N\in[n,\infty)$
and  $Diam(M)=\frac{\pi}{\sqrt{k}}$, then $(M,F)$ is isometric to a standard Randers
sphere.
\end{theorem}

\begin{proof}
Suppose that the Randres metric $F$ is given by
\begin{align}\label{10}
F=\frac{\sqrt{\lambda h^2+W_0^2}}{\lambda}-\frac{W_0}{\lambda},\quad W_0=W_iy^i,
\end{align}
where $h=\sqrt{h_{ij}(x)y^iy^j}$ is a Riemannian metric, $W=W^i\frac{\partial}{\partial x^i}$ is
a vector field on $M$, and
 $$W_i=h_{ij}W^j,\quad \lambda:=1-W_i W^i=1-h(x,W)^2.$$
Under the condition of Theorem 2.8, it follows from Theorem 2.1  that
the flag curvature of $F$ is $K=k$ and $S=0$. First, according to Theorem 1.1 in \cite{O2},
only (constant multiplications of) the Busemann-Hausdorff measures can satisfy $S\equiv0$ on Randers spaces.
Thus we might as well suppose that $d\mu$ is the Busemann-Hausdorff volume form and $S_{BH}=0$, which
 is an equivalence that $W$ is a Killing vector fields
on $M$. Second, for a Randers metric $F$ expressed above, it follows from \cite{BCS} that $F$ has
constant flag curvature $K=k$ if and only if $h$ has constant sectional curvature $K_h=k+c^2$
and $S_{BH}=(n+1)cF$. Thus $K_h=k$. By Theorem 2.1, $M$ is homeomorphic to $\mathbb{S}^n$.
As a result, $(M,h)$ is a compact simply connected Riemannian manifold of
sectional curvature $K_h=k$.  Therefore, $(M,h)$ is isometric to the Euclid sphere
$(\mathbb{S}^n(\frac{1}{\sqrt{k}}),\mathfrak{g})$, where $\mathfrak{g}$ denotes the standard sphere metric on
$\mathbb{S}^n(\frac{1}{\sqrt{k}})$. Now the Randers  metric $F$ defined on $\mathbb{S}^n(\frac{1}{\sqrt{k}})$ is given by
$$F=\frac{\sqrt{\lambda \mathfrak{g}^2+W_0^2}}{\lambda}-\frac{W_0}{\lambda},$$
where $W$ is a Killing vector field on $(\mathbb{S}^n(\frac{1}{\sqrt{k}}),\mathfrak{g})$.
\end{proof}
\vspace{3mm}

\hspace{-4mm}\emph{Proof of Theorem 0.1}.
The first part of Theorem 0.1 follows from Corollary 2.3 directly.  Notice that Theorem 0.4 in \cite{ZH},
 shows that,  a regular non-Randers $(\alpha, \beta)$-metric  with isometric $S$-curvature and
scalar flag curvature on a Finsler $n$-manifold $(n\geq3)$ must be a Minkowski metric. Combining this with Theorem 2.8
the second part of Theorem 0.1 follows.

$\hspace{120mm}\square$

\section {Some applications on the first eigenvalue}
In this section, we use the maximum diam theorem to describe the rigidity of the Finsler manifolds
on which the first eigenvalue attains its lower bound. First we give the following:
\begin{lemma}
Let $(\mathfrak{S}^n(\frac{1}{\sqrt{k}}),F,d\mu)$ be a standard Finsler sphere, and $r(x)=d_F(p,x)$ be the distance function from a fixed
 point $p\in\mathfrak{S}^n$. Then 
$$\tilde{f}=-\cos(\sqrt{k}r),\quad 0\leq r\leq\frac{\pi}{\sqrt{k}}$$
is a 1-st eigenfunction with $\lambda_1= nk$.
\end{lemma}
\begin{proof}
For a standard Finsler sphere $\mathfrak{S}^n(\frac{1}{\sqrt{k}})$,
it follows from  the proof of Theorem 2.1 that $\Delta r=(n-1)\sqrt{k}\cot(\sqrt{k}r)$.
Noticed that the volume form satisfies (see also the proof of Theorem 2.1)
$\frac{\partial}{\partial r}\log\sigma_p(r,\theta)=\frac{\partial}{\partial r}\log\tilde{\sigma}(r)$, which yields
$$\frac{\sigma_p(R,\theta)}{\tilde{\sigma}(R)}=\frac{\sigma_p(r,\theta)}{\tilde{\sigma}(r)}:=C(\theta),
\quad r\leq R\leq\frac{\pi}{\sqrt{k}},$$
where $\tilde{\sigma}(r)=(\frac{\sin(\sqrt{k}r)}{\sqrt{k}})^{n-1}$.
Therefore,
\begin{align}
\int_{\mathfrak{S}^n}\tilde{f}d\mu&=\int_0^{\frac{\pi}{\sqrt{k}}}\int_{S_pM}\tilde{f}\sigma_p(r,\theta)drd\theta
=\int_0^{\frac{\pi}{\sqrt{k}}}\int_{S_pM}\tilde{f}C(\theta)\tilde{\sigma}(r)drd\theta\nonumber\\
&=-\int_0^{\frac{\pi}{\sqrt{k}}}\cos(\sqrt{k}r)(\frac{\sin(\sqrt{k}r)}{\sqrt{k}})^{n-1}dr\int_{S_pM}C(\theta)d\theta=0.\nonumber
\end{align}
Moreover, $\nabla \tilde{f}=\sqrt{k}\sin(\sqrt{k}r)\nabla r,0< r<\frac{\pi}{\sqrt{k}}$, which means
that $\nabla \tilde{f}$  and $\nabla r$ have the same direction. Thus
\begin{align}
\Delta \tilde{f}&=\tilde{f}'\Delta r+\tilde{f}''=\sqrt{k}\sin(\sqrt{k}r)\times (n-1)\sqrt{k}\cot(\sqrt{k}r)+k\cos(\sqrt{k}r)\nonumber\\
&=nk\cos(\sqrt{k}r)=-nk\tilde{f}.\nonumber
\end{align}
Therefore, we obtain the first eigenfunction $\tilde{f}$ of $(\mathfrak{S}^n,F,d\mu)$.

\end{proof}

\begin{theorem}
Let $(M,F,d\mu)$ be a complete connected Finsler n-manifold with the Busemann-Hausdorff volume form.
If the weighted Ricci curvature satisfies
$\emph{Ric}_n\geq(n-1)k>0$, then the first eigenvalue of Finsler-Laplacian
$$\lambda_1\geq nk.$$
The equality holds if and only if $(M, F)$ is isometric to a standard Finsler sphere.
\end{theorem}

\begin{proof}
The estimate of the first eigenvalue is proved in \cite{YHS1}. If the equality holds, we deduce that
$Diam(M)=\frac{\pi}{\sqrt{k}}$(\cite{YHS1}). Then by Corollary 2.3, $(M, F)$ is isometric to
$\mathfrak{S}^n(\frac{1}{\sqrt{k}})$. Conversely, it follows from Lemma 3.1.
\end{proof}

\begin{theorem}
Let $(M,F,d\mu)$ be a complete connected Randers n-manifold.
If the weighted Ricci curvature satisfies $\emph{Ric}_N\geq(N-1)k>0$ for
some real number $N\in[n,\infty)$,
then the first eigenvalue of Finsler-Laplacian
$$\lambda_1\geq Nk.$$
Moreover, if the volume form satisfies
$\lim\limits_{r\to0}\int_{S_zM}\frac{\sigma_z(r,\theta)}{r^{N-1}}d\theta=C, \forall z\in M$,
then the equality holds if and only if  $(M,F)$ is isometric to a standard Randers
sphere.
\end{theorem}

\begin{remark}
Theorem 3.2 shows that the lower bound of the first eigenvalue of Finsler-Laplacian can be attained
in a standard Finsler sphere. However, we are  still unable to characterize the sphere  in more details.
In Theorem 3.3, we narrow the scope to Randers manifolds. This is because
Randers spheres are completely clear even though their quantity is also infinite.
\end{remark}

\begin{proof}
If the equality holds, we have $Diam(M)=\frac{\pi}{\sqrt{k}}$ (\cite{YHS1}). Then the conclusion follows from
 Theorem 2.8 directly.

To prove the reverse side, we  point out that, by Proposition 2.5, for a standard Randers sphere
$\mathcal{S}^n(\frac{1}{\sqrt{k}})$, $K=k,S=0$ and $\textmd{Ric}_N=\textmd{Ric}=(n-1)k$,
and thus $N=n,\lambda_1\geq nk$. Therefore, we only need to construct a Randers metric on the
sphere $\mathbb{S}^n(\frac{1}{\sqrt{k}})$ such that the first eigenvalue of Finsler-Laplacian
attains its lower bound $nk$. In  Lemma 3.1 we have obtained the first eigenfunction
$\tilde{f}$. In the following, we want to give another 1-st eigenfunction via a different method.

Let $p,q$ be the north pole and south pole of the Euclidean sphere $(\mathbb{S}^n(\frac{1}{\sqrt{k}}),\mathfrak{g})$, respectively.
Let $\varphi(t,x)$ be the rotation transform on $\mathbb{S}^n(\frac{1}{\sqrt{k}})$ satisfying
$\varphi(t,p)=p$ and $\varphi(t,q)=q$ for any $t$. Then $\varphi(t,x)$ is a isometric
transform on $\mathbb{S}^n(\frac{1}{\sqrt{k}})$ and $X=\frac{\partial\varphi(t,x)}{\partial t}$
is a Killing vector field. It is easy to see that $X\bot\nabla^{\mathfrak{g}} \rho$ where $\rho(x)=d_\mathfrak{g}(p,x)$
is the distance function and $\nabla^{\mathfrak{g}} \rho$ is the gradient with respect to $\mathfrak{g}$.
On the other hand, it is well known that the first eigenfunction $f$ of $(\mathbb{S}^n(\frac{1}{\sqrt{k}}),\mathfrak{g})$
is radial function, i.e., $f(\rho,\theta)=f(\rho)$, where $\rho(x)=d_\mathfrak{g}(p,x)$. Thus, we have
$$X(f)=0.$$
Note that the volume form of a Rander sphere $(\mathcal{S}^n(\frac{1}{\sqrt{k}}),F,d\mu)$
is the Busemann-Hausdorff volume form.
Therefore, $d\mu=dV_{\mathfrak{g}}$ which yields
\begin{align}\label{11}
\int_M|f|^2d\mu=\int_M|f|^2dV_{\mathfrak{g}}.
\end{align}
 Recall that the dual metric of  (\ref{10}) is
$$F^{\ast}:=h^{\ast}+W^{\ast}=\sqrt{h^{ij}\xi_{i}\xi_{j}}+W^{i}\xi_{i},$$
 where $(h^{ij})=(h_{ij})^{-1}$ and $W^{i}=W_{j}h^{ij}$.
Thus, for a $C^{1}$ function $f$, we have
$$F(\nabla f)=F^{\ast}(df)=h^{\ast}(df)+W^{i}f_{i}=h(\nabla^{h} f)+W(f),$$
If $f$ is the first eigenfunction  of $(\mathbb{S}^n(\frac{1}{\sqrt{k}}),\mathfrak{g})$ and $W=X$
is the Killing vector as above, then we have $F(\nabla f)=\mathfrak{g}(\nabla^{\mathfrak{g}} f)$, which gives
\begin{align}\label{12}
\int_MF(\nabla f)^2d\mu=\int_M\mathfrak{g}(\nabla^{\mathfrak{g}} f)^2dV_{\mathfrak{g}}.
\end{align}
Combining (\ref{11}) and (\ref{12}), and noting that the first eigenvalue of
$(\mathbb{S}^n(\frac{1}{\sqrt{k}}),\mathfrak{g})$ is $nk$, we obtain
$$\lambda_1\leq\frac{\int_MF(\nabla f)^2d\mu}{\int_M|f|^2d\mu}=
\frac{\int_M\mathfrak{g}(\nabla^{\mathfrak{g}} f)^2dV_{\mathfrak{g}}}{\int_M|f|^2dV_{\mathfrak{g}}}
=nk.$$
On the other hand, we know from the first assertion of Theorem 3.2 that $\lambda_1\geq nk$.
Thus $\lambda_1=nk$.
This implies that $f$ is also the first eigenfunction  of the Randers sphere
$(\mathcal{S}^n(\frac{1}{\sqrt{k}}),F)$.
\end{proof}

The existence of the first eigenfunction of Finsler Laplacian is proved by \cite{GS}.
But so far, any explicit 1-st eigenfunction has not been found in non-Riemannian case.
It seems very difficult to do the computation since Finsler Laplacian is
a nonlinear operator.

\begin{remark}
In the proof of Lemma 3.1, we first construct an explicit 1-st eigenfunction of
Finsler Laplacian $\tilde{f}$ on the sphere $(\mathfrak{S}^n,F,d\mu)$,
which means that $-\tilde{f}$ is the first eigenfunction of the reverse sphere
$(\mathfrak{S}^n,\overleftarrow{F},d\mu)$.
 In particular, in a standard Randers sphere $(\mathcal{S}^n(\frac{1}{\sqrt{k}}),F,d\mu)$, we give two 1-st eigenfunctions
$\tilde{f}$ and $f$ (see the proof of Theorem 3.3) since they are not necessarily equal.
\end{remark}
\vspace{3mm}

\hspace{-4mm}\emph{Proof of Theorem 0.2}.\quad
It follows directly from Theorems 3.2-3.3 and Theorem 0.4 in \cite{ZH}.

$\hspace{120mm}\square$

\vspace{3mm}
\bibliographystyle{amsplain}

\end{document}